\documentclass[letterpaper, 11pt]{article}
\usepackage[numbers,sort&compress]{natbib}
\usepackage[ruled,vlined]{algorithm2e}
\usepackage{verbatim}
\usepackage{amsmath,amssymb}
\usepackage{xcolor} 
\usepackage[ruled,vlined]{algorithm2e}
\usepackage{hyperref}
\usepackage[capitalise]{cleveref}
\usepackage[T1]{fontenc}
\usepackage[utf8]{inputenc}
\usepackage[margin=1in]{geometry}
\usepackage{graphicx,booktabs,array,microtype}
\usepackage{amsthm}
\usepackage{lineno}
\usepackage{setspace}
\newtheorem{theorem}{Theorem}
\newtheorem{corollary}{Corollary}
\theoremstyle{definition}
\newtheorem{definition}{Definition}
\newtheorem{remark}{Remark}
\newcolumntype{R}{@{\extracolsep{\fill}}r}
\newcolumntype{C}{@{\extracolsep{\fill}}c}
\newcommand{\tblwidth}{\linewidth}

\begin{document}

\let\WriteBookmarks\relax

%\begin{frontmatter}

\title{Smooth periodic extensions by extrapolation and blending for FFT-based differentiation}

\author{Iv\'{a}n M\'{e}ndez-Cruz
\qquad Faisal Amlani$^*$\\[.5em]
\small Universit\'{e} Paris-Saclay, CentraleSup\'{e}lec, ENS Paris-Saclay, CNRS, \\[-.25em]
\small LMPS - Laboratoire de M\'{e}canique Paris-Saclay, Gif-sur-Yvette, France}
\date{}
% \affiliation[1]{organization={LMPS, ENS Paris-Saclay},
%                 addressline={4 Avenue des Sciences}, 
%                 city={Gif-sur-Yvette},
% %               citysep={}, % Uncomment if no comma needed between city and postcode
%                 postcode={91190},
%                 state={\^{I}le-de-France},
%                 country={France}}

\maketitle
\begin{abstract}
\noindent Non-periodic functions on closed intervals are extended into sufficiently smooth periodic functions through a continuous Hermite interpolation framework with necessary and sufficient endpoint conditions, using smooth blend-to-zero operators and short extrapolation recurrences. A matrix-free method is introduced for generating these smooth periodic extensions on the fly directly from function values at equidistant points. No precomputed continuation operator or auxiliary basis construction is required. The resulting extensions can then be used for FFT-based differentiation of non-periodic functions. Numerical tests are provided to show the effectiveness of the method, including for the highly accurate solution of a 1D linear advection equation.
\end{abstract}

%%%%%%%%%%%%%%%%%%%%%%%%%%%%%%%%%%%%%%% new section %%%%%%%%%%%%%%%%%%%%%%%%%%%%%%%%%%%%%%%%%%
\section{Introduction}\label{sec:intro}
 
Inspired by Fourier continuation (FC) methods~\cite{Boyd2002,Huybrechs2010, bruno2010, AlbinPathmanathan2014, AmlaniBruno2016, amlani2023fourier, NainwalAnand2025, NainwalAnand2026} and blending techniques~\cite{Boyd2006,Laksa2022}, this paper introduces a method for generating $C^n$-periodic extensions of non-periodic functions on closed intervals. Existing FC methods include constructions based on precomputed continuation operators and Gram-polynomial bases~\cite{bruno2010,AmlaniBruno2016}, Gram-polynomial constructions with explicit Hermite continuations~\cite{NainwalAnand2025,NainwalAnand2026}, and discrete periodic extensions that blend polynomial extrapolants by an approximate step function computed offline~\cite{AlbinPathmanathan2014}. Here, the continuation values are generated directly from equidistant samples using short extrapolation recurrences and an explicit smooth step function with exactly flat ends. No precomputed continuation matrix or Gram-polynomial construction is required. Consequently, the extension can be generated with parameters selected on the fly (e.g., at every time step of a time-dependent PDE solver). Necessary and sufficient conditions for such extensions are provided in \Cref{sec:Hermite} by (not necessarily polynomial) two-point Hermite interpolation. Based on~\cite{MendezFaisal2026}, the Hermite interpolants are generated by blend-to-zero operators in \Cref{sec:operators} using smooth step functions, while in \Cref{sec:extrapolation} the unknown values on the extended interval are generated by extrapolation based on low-degree polynomial interpolants. In \Cref{sec:method}, blending and extrapolation are combined in a method for generating $C^n$-periodic extensions at equidistant points. In \Cref{sec:differentiation}, the derivatives of non-periodic functions are approximated via fast Fourier transform (FFT)-based differentiation of $C^n$-periodic extensions and numerical experiments illustrate the use of the new method for numerical differentiation. In particular, a highly accurate numerical solution of a 1D advection equation is computed on a non-periodic domain.

%%%%%%%%%%%%%%%%%%%%%%%%%%%%%%%%%%%%%%% new section %%%%%%%%%%%%%%%%%%%%%%%%%%%%%%%%%%%%%%%%%%
\section{Construction of smooth periodic extensions using Hermite interpolation}
\label{sec:Hermite}

Real-valued non-periodic functions are extended periodically into sufficiently smooth functions.

\begin{definition}
  Let $\ell$ be a positive number, $m,n\in\mathbb{N}$ such that $n\leq m$, and $f\in C^{m}[a,b]$. 
  A periodic function $F\in C^{n}(\mathbb{R})$ with period $P=b-a+\ell$ such that $F(x)=f(x)$ for $a\leq x\leq b$ is called a {\em $C^n$-periodic extension of $f(x)$}.
\end{definition}

Necessary conditions for $C^{n}$-periodic extensions are given by Hermite interpolation conditions at the endpoints.

\begin{theorem}\label{thm1}
 Let $\ell$ be a positive number, $m,n\in\mathbb{N}$ such that $n\leq m$. 
 Suppose that $F(x)$ is a $C^{n}$-periodic extension of $f(x)$ with period $P=b-a+\ell$. 
 Define $f_{H}(x)=F(x)$ for $b\leq x\leq b+\ell$. 
 Then
  \begin{equation}\label{eq:F}
    f_{H}^{(k)}(b) = f^{(k)}(b),
    \quad f_{H}^{(k)}(b+\ell) = f^{(k)}(a),
    \quad k=0,\dots,n.
  \end{equation}
\end{theorem}

\begin{proof}
  
Let $k\in\{0,\dots,n\}$. Since $F\in C^{n}(\mathbb{R})$ and $F$ agrees with $f$ on $[a,b]$, we have $F^{(k)}(a)=f^{(k)}(a)$ and $f_{H}^{(k)}(b)=F^{(k)}(b)=f^{(k)}(b)$. Since $F$ is $P$-periodic, $F^{(k)}$ is also $P$-periodic. Hence, using $a+P=b+\ell$, it follows that $f_{H}^{(k)}(b+\ell)=F^{(k)}(a+P)=F^{(k)}(a)=f^{(k)}(a)$. \qed

% \qed
\end{proof}

\cref{eq:F} is also a sufficient condition for generating $C^{n}$-periodic extensions.

\begin{theorem}\label{thm2}
  Let $\ell$ be a positive number, $f\in C^{n}[a,b]$ and $f_{H}\in C^{n}[b,b+\ell]$.  
  Suppose that \cref{eq:F} holds. Set
  \begin{equation}\label{eq:ftilde}
    f_{E}(x) := \begin{cases}
      f(x), & a\leq x < b;\\
      f_{H}(x), & b\leq x \leq b+\ell,
    \end{cases}  \quad\text{and}\quad
    s(x):= x - \left\lfloor\frac{(x-a)}{P}\right\rfloor P,
  \end{equation}
  and define $F(x) := f_{E}(s(x))$ for $x\in\mathbb R$, where $P=b-a+\ell$.
  Then $F(x)$ is a $C^{n}$-periodic extension of $f(x)$.
\end{theorem}

\begin{proof}
The function $F$ is well defined and agrees with $f$ on $[a,b]$. Since $f\in C^{n}[a,b]$, $f_{H}\in C^{n}[b,b+\ell]$, and \cref{eq:F} gives matching derivatives at $x=b$, we have $f_{E}\in C^{n}[a,a+P]$. Away from the points $A=\{a+jP:j\in\mathbb{Z}\}$, we have $F\in C^{n}$ and $F^{(k)}(x)=f_{E}^{(k)}(s(x))$. At each point of $A$, the left and right limits of $F^{(k)}$ are $f_{E}^{(k)}(a+P)$ and $f_{E}^{(k)}(a)$, respectively, which are equal by \cref{eq:F} for $k=0,\dots,n$. The equality for $k=0$ gives continuity of $F$. Applying the mean value theorem successively to $F,F',\dots,F^{(n-1)}$ then shows that the derivatives through order $n$ exist and are continuous at each point of $A$. Hence, $F\in C^{n}(\mathbb{R})$. Finally, $s(x+P)=s(x)$, so $F(x+P)=F(x)$. \qed
\end{proof}

%%%%%%%%%%%%%%%%%%%%%%%%%%%%%%%%%%%% new section %%%%%%%%%%%%%%%%%%%%%%%%%%%%%%%%%%%%%%%%%%%%
\section{Blend-to-zero operators and smooth step functions}
\label{sec:operators}

Functions that satisfy \cref{eq:F} can be generated by operators that preserve function values and derivatives at one endpoint while setting those at the other endpoint to zero. These operators are defined in~\cite{MendezFaisal2026} along with smooth step functions and smooth transition functions. An operator $\mathcal{B}: C^{q}[a_{0},b_{0}]\to C^{q}[a_{0},b_{0}]$ is a \emph{leftward blend-to-zero operator of order} $n\leq q$ if for all $\varphi\in C^{q}[a_{0},b_{0}]$
the function $\varphi_{0}=\mathcal{B}(\varphi)$ satisfies the following:
  \begin{equation}
    \varphi_{0}^{(k)}(a_{0}) = 0, \quad 
    \varphi_{0}^{(k)}(b_{0}) = \varphi^{(k)}(b_{0}), \quad k=0,\dots,n.
  \end{equation}

Let $f,g\in C^{q}[a_{0},b_{0}]$ and $n\leq q$. A function $h_{0}\in C^{n}[a_{0},b_{0}]$ is a \emph{$C^{n}$-transition function} from $f$ to $g$ on $[a_{0},b_{0}]$ if
  \begin{equation}\label{eq:transition}
    h_{0}^{(k)}(a_{0}) = f^{(k)}(a_{0}), \quad 
    h_{0}^{(k)}(b_{0}) = g^{(k)}(b_{0}), \quad k=0,\dots,n.
  \end{equation}

 A real-valued function $\sigma\in C^{n}[a,b]$ has \emph{flat ends} of order $n$ if $\sigma^{(k)}(a)=\sigma^{(k)}(b)=0$ for $k=1,\dots,n$. A strictly increasing function $\sigma(x)$ on $[0,1]$ with flat ends  of order $n$ such that $\sigma(0)=0$ and $\sigma(1)=1$ is called a \emph{smooth step function} of order $n$~\cite{MendezFaisal2026}. 
 
Note that a smooth step function $\sigma(x)$ is symmetric about the point $(1/2,1/2)$ iff $\sigma(x)+\sigma(1-x)=1$ for $0\leq x\leq 1$. A family of smooth step functions of order $n=2m+1$ that are symmetric about the point $(1/2,1/2)$ is given by trigonometric B-functions $T_{n}(x)=1/2- 1/(2a_{m})\sum_{j=0}^{m}a_{m,j}\cos((2j+1)\pi x)$, where $a_{m,j}= (-1)^{m-j} \binom{2m+1}{m-j} /(2j+1)$ and $a_{m}=\sum_{j=0}^{m}a_{m,j}$~\cite{MendezFaisal2026,Olofsen2019}. See~\cite{Laksa2022} for other families of (symmetric) smooth step functions.

Leftward blend-to-zero operators can be generated by smooth step functions. Smooth transition functions can then be generated using these blend-to-zero operators.

\begin{theorem}[\protect{\cite[Corollary~5]{MendezFaisal2026}}]\label{thm:BlendStep}
  Let $\sigma(x)$ be a smooth step function of order $n$. 
  Set $\mathcal{B}(\varphi)=(\sigma\circ\lambda)\cdot\varphi$ for $\varphi\in C^{q}[a_{0},b_{0}]$, where $\lambda(x)=(x-a_{0})/(b_{0}-a_{0})$ for $a_{0}\leq x\leq b_{0}$ and $n\leq q$. 
  Then $\mathcal{B}$ is a linear leftward blend-to-zero operator of order $n$.
\end{theorem}

\begin{theorem}[\protect{\cite[Corollary~1]{MendezFaisal2026}}]\label{thm:BlendHermite}
 Let $f,g\in C^{q}[a_{0},b_{0}]$. Suppose that $\mathcal{B}: C^{q}[a_{0},b_{0}]\to C^{q}[a_{0},b_{0}]$ is a leftward blend-to-zero operator of order $n\leq q$. 
 Set $f_{0} := f-\mathcal{B}(f)$, $g_{0} := \mathcal{B}(g)$, and $h_{0}(x) := f_{0}(x)+g_{0}(x)$. 
 Then $h_{0}(x)$ is a $C^{n}$-transition function from $f(x)$ to $g(x)$ on $[a_{0},b_{0}]$.
\end{theorem}

Following~\cite{NainwalAnand2026,AlbinPathmanathan2014,Boyd2006}, a smooth function $f(x)$ on $[a,b]$ can be extended on both sides and blended by smooth step functions to generate a smooth transition between $f(x)$ and its translated copy. Thus, the construction of a $C^{n}$-periodic extension can be regarded as the generation of a $C^{n}$-transition function.

\begin{corollary}\label{cor:fB}
 Let $\ell>0$, $f_{L}\in C^{n}[a-\ell,a]$, $f_{R}\in C^{n}[b,b+\ell]$, and $\sigma$ a smooth step function of order $n$. Set 
 \begin{equation}
  f_{H}(x) := f_{R}(x) - \sigma(\lambda(x)) f_{R}(x) + \sigma(\lambda(x)) f_{L}(x+a-b-\ell), \quad b\leq x\leq b+\ell,
 \end{equation} 
  where $\lambda(x):=(x-b)/\ell$. 
  Then $f_{H}^{(k)}(b) = f_{R}^{(k)}(b)$ and 
  $f_{H}^{(k)}(b+\ell) = f_{L}^{(k)}(a)$ for $k=0,\dots,n$. 
\end{corollary}

\begin{proof}
Apply \Cref{thm:BlendStep,thm:BlendHermite} on $[b,b+\ell]$ with $q=n$, $f=f_{R}$, and $g(x)=f_{L}(x+a-b-\ell)$. Since $f_{H}=f_{R}-\mathcal{B}(f_{R})+\mathcal{B}(g)$ and $g^{(k)}(b+\ell)=f_{L}^{(k)}(a)$ for $k=0,\dots,n$, the result follows. \qed
\end{proof}

\begin{remark}\label{rmk:Taylor}
  Let $\sigma(x)$, $\lambda(x)$, $f_{L}(x)$, $f_{R}(x)$ and $f_{H}(x)$ be as in \Cref{cor:fB}, let $F(x)$ be as in \Cref{thm2}, and $T_{L}(x)$ and $T_{R}(x)$ be the $n$-th Taylor polynomials of $f(x)$ about $x=a$ and $x=b$, respectively. Set $f_{L}(x) = T_{L}(x)$ and $f_{R}(x) = T_{R}(x)$. From \Cref{cor:fB} and the fact that $T_{L}^{(k)}(a) = f^{(k)}(a)$ and $T_{R}^{(k)}(b) = f^{(k)}(b)$ for $k=0,\dots,n$, it is deduced that $f_{H}(x)$ satisfies \cref{eq:F}. Then, from \Cref{thm2}, it follows that $F(x)$ is a $C^{n}$-periodic extension of $f(x)$.
\end{remark}

%%%%%%%%%%%%%%%%%%%%%%%%%%%%%%%%%%%% new section %%%%%%%%%%%%%%%%%%%%%%%%%%%%%%%%%%%%%%%%%%%%
\section{Extrapolation via Lagrange polynomials}
\label{sec:extrapolation}

In the discrete setting, given $f_j=f(x_j)$ at the equidistant points $x_j=a+(j-1)\Delta x$, $j=1,\dots,N$, the problem is to generate the continuation values $F(x_{N+q})$ at $x_{N+q}=b+q\Delta x$, $q=1,\dots,C$, where $\Delta x=(b-a)/(N-1)$. The Taylor construction in \Cref{rmk:Taylor} requires higher-order endpoint derivatives, whereas in the discrete setting only the samples $f_j=f(x_j)$ are available. We therefore extrapolate $f(x)$ using interpolating polynomials.

Given positive integers $C$ and $d$ such that $d\leq N$, denote by $p_{L}(x)$ and $p_{R}(x)$ the polynomials of degree at most $d-1$ such that $p_{L}(x_{j}) = f_j$ and $p_{R}(x_{N+1-j}) = f_{N+1-j}$ for $j=1,\dots,d$. New point values at $x_{1-q} = a-q\Delta x$ and $x_{N+q} = b+q\Delta x$ are generated by taking $f_{1-q} = p_{L}(x_{1-q})$ and $f_{N+q} = p_{R}(x_{N+q})$ for $q=1,\dots,C$. Direct evaluations of $p_{L}(x)$ and $p_{R}(x)$ beyond the interpolation stencil can be sensitive to rounding errors~\cite{Webbetal2012}. Here, the extrapolated values are instead generated recursively using recurrence relations.

\begin{theorem}\label{thm:extrapolation}
Let $z_{0}<\cdots<z_{d+1}$ be equidistant points,  $g_{1},\dots,g_{d}\in\mathbb{R}$, and let $p(x)$ be the polynomial of degree at most $d-1$ such that $p(z_{j}) = g_{j}$ for $j=1,\dots,d$. 
Set $w_{j} = (-1)^{j-1}\binom{d}{j}$ for $j=1,\dots,d$. 
Then 
\begin{align}
  p(z_{0}) & = \textstyle\sum_{j=1}^{d} w_{j}g_{j} \label{eq:q0}, \\
  p(z_{d+1}) & = \textstyle\sum_{j=1}^{d} w_{j}g_{d-j+1}. \label{eq:q1}
\end{align}
\end{theorem}

\begin{proof}
Since $p$ has degree at most $d-1$, its $d$-th forward difference vanishes: $\sum_{j=0}^{d}(-1)^{d-j}\binom{d}{j}p(z_{j})=0$. Solving for $p(z_{0})$ and using $p(z_{j})=g_{j}$ gives \cref{eq:q0}. Applying the same to $z_{1},\dots,z_{d+1}$ and solving for $p(z_{d+1})$ gives \cref{eq:q1}. \qed
\end{proof}

  Since the extrapolated values lie on $p_{L}$ and $p_{R}$ by construction, successively applying \cref{eq:q0,eq:q1} to $p_{L}$ and $p_{R}$ on stencils shifted by one point at a time gives
\begin{equation}\label{eq:fprev}
  f_{1-q}=\textstyle\sum_{j=1}^{d}w_{j}f_{1+j-q}, \qquad
  f_{N+q}=\textstyle\sum_{j=1}^{d}w_{j}f_{N+q-j}, \qquad q=1,\dots,C.
\end{equation}
The recurrence coefficients depend only on $d$ and are independent of $\Delta x$.

%%%%%%%%%%%%%%%%%%%%%%%%%%%%%%%%%%%% new section %%%%%%%%%%%%%%%%%%%%%%%%%%%%%%%%%%%%%%%%%%%%
% \enlargethispage{1\baselineskip}
\section{An explicit extrapolation and blending method for smooth periodic extensions}
\label{sec:method}

The extrapolation and blending constructions are now combined to generate a $C^{d}$-periodic extension of an approximation of $f(x)$ on $[a,b]$. Let $f\in C^{d}[a,b]$, and let $C,d$ and $N$ be positive integers such that $2d-1\leq N$. Let $\sigma(x)$ be a smooth step function of order at least $d$. Let $x_{j}=a+(j-1)\Delta x$ and $f_{j}=f(x_{j})$ for $j=1,\dots,N$, where $\Delta x=(b-a)/(N-1)$.
Define
\begin{align}
  \phi_{L}(x) & = \left(1-\sigma\left(\frac{x-a}{x_{d}-a}\right)\right)p_{L}(x) + 
  f(x)\sigma\left(\frac{x-a}{x_{d}-a}\right),
  & \quad a\leq x\leq x_{d}, \\
  \phi_{R}(x) & = \left(1-\sigma\left(\frac{x-x_{N-d+1}}{b-x_{N-d+1}}\right)\right)f(x) + 
  p_{R}(x)\sigma\left(\frac{x-x_{N-d+1}}{b-x_{N-d+1}}\right),
  & \quad x_{N-d+1}\leq x\leq b,
\end{align}
and approximate $f(x)$ by the piecewise function
\begin{equation}
  \widetilde{f}(x) := \begin{cases}
      \phi_{L}(x), & a\leq x < x_{d}; \\
      f(x), & x_{d}\leq x\leq x_{N-d+1}; \\
      \phi_{R}(x), & x_{N-d+1}< x\leq b.
  \end{cases}
\end{equation}
From \cref{thm:BlendStep,thm:BlendHermite} it follows that $\phi_{L}(x)$ is a $C^{d}$-transition function from $p_{L}(x)$ to $f(x)$ on $[a,x_{d}]$ and that $\phi_{R}(x)$ is a $C^{d}$-transition function from $f(x)$ to $p_{R}(x)$ on $[x_{N-d+1},b]$. Hence, $\widetilde{f}\in C^{d}[a,b]$, with $\widetilde{f}^{(k)}(a)=p_{L}^{(k)}(a)$ and $\widetilde{f}^{(k)}(b)=p_{R}^{(k)}(b)$ for $k=0,\dots,d$. Furthermore, $\widetilde{f}(x_{j})=f_{j}$ for $j=1,\dots,N$.

Set $\ell=C\Delta x$. Applying \Cref{cor:fB} with $n=d$, $f_{L}=p_{L}$, and $f_{R}=p_{R}$, we have
\begin{equation}\label{eq:fH2}
  f_{H}(x) = \left(1 - \sigma\left((x-b)/\ell\right)\right) p_{R}(x) + 
  \sigma\left((x-b)/\ell\right) p_{L}(x+a-b-\ell).
\end{equation}
Noting that $(x_{N+q}-b)/\ell = q/C$ and $x_{N+q}+a-b-\ell=x_{1-C+q}$ and recalling that $f_{1-C+q} = p_{L}(x_{1-C+q})$ and $f_{N+q} = p_{R}(x_{N+q})$ it follows from \Cref{eq:fH2} that
\begin{equation}\label{eq:fH3}
  f_{H}(x_{N+q}) = \left(1-\sigma(q/C)\right) f_{N+q} + \sigma(q/C) f_{1-C+q},
  \quad q=0,\dots,C.
\end{equation}
In addition, if $\sigma(x)$ is symmetric about the point $(1/2,1/2)$, then from \Cref{eq:fH3} it is deduced that
\begin{equation}\label{eq:pisym} 
  f_{H}(x_{N+q}) = \sigma(1-q/C)f_{N+q} + \sigma(q/C)f_{1-C+q}, \quad q=0,\dots,C.
\end{equation}

Let $\widetilde{F}(x) :=\widetilde{f}_{E}(s(x))$ for $x\in\mathbb{R}$, where $\widetilde{f}_{E}(x) := \widetilde{f}(x)$ on $a\leq x\leq b$ and $\widetilde{f}_{E}(x) := f_{H}(x)$ on $b< x \leq b+\ell$. From \Cref{thm2} it follows that $\widetilde{F}(x)$ is a $C^{d}$-periodic extension of $\widetilde{f}(x)$ with period $P=b-a+C\Delta x=(N+C-1)\Delta x$.

Combining \cref{eq:fprev,eq:pisym} generates the continuation values $f_{H}(x_{N+q})$, $q=1,\dots,C$, on the fly directly from the equidistant samples $f_{j}$, $j=1,\dots,N$. This \emph{Extrapolation \& Blending} (EB) construction is summarized in \Cref{alg:extrablend}.

\begin{algorithm}[!htbp]
  \small
  \caption{Extrapolation \& Blending (EB)}
  \label{alg:extrablend}
  \DontPrintSemicolon 
    \KwIn{$C,d,N\in\mathbb{N}$ with $2d-1\leq N$; a symmetric smooth step function $\sigma(x)$ of order at least $d$; $f_{j}$ for $j=1,\dots,N$}
  \KwOut{$f_{H}(x_{N+q})$  given by \cref{eq:pisym} for $q=1,\dots,C$}
  $\sigma_{q} \gets \sigma(q/C)$ for $q=1,\dots,C-1$\;
  $w_{j} \gets (-1)^{j-1}\binom{d}{j}$ for $j=1,\dots,d$\; 
  $f_{q}^{H} \gets 0$ for $q=1,\dots,C$\;
  \For{$q=1$ to $C-1$}{
    $f_{1-q} \gets \sum_{j=1}^{d} w_{j}f_{j+1-q}$\;
    $f_{N+q} \gets \sum_{j=1}^{d} w_{j}f_{N+q-j}$\;
    $f_{q}^{H} \gets f_{q}^{H} + \sigma_{C-q}f_{N+q}$\;
	  $f_{C-q}^{H} \gets f_{C-q}^{H} + \sigma_{C-q}f_{1-q}$\;
  }
  $f_{C}^{H} \gets f_{1}$
\end{algorithm}

\begin{remark}
With the same degree-$(d-1)$ boundary interpolants and $C$ continuation values, EB reproduces the discrete extension of~\cite{AlbinPathmanathan2014} in exact arithmetic when $\sigma(q/C)$ equals their sampled approximate step values. Unlike~\cite{AlbinPathmanathan2014}, EB uses an explicit smooth step function and requires neither an offline approximate step function nor a continuation operator.
\end{remark}

%%%%%%%%%%%%%%%%%%%%%%%%%%%%%%%%%%%% new section %%%%%%%%%%%%%%%%%%%%%%%%%%%%%%%%%%%%%%%%%
\section{Numerical experiments with FFT-based differentiation}
\label{sec:differentiation}

Consider the problem of numerical differentiation of the real-valued function $f(x)$ on $[a,b]$ tabulated at the equidistant points $x_{j}=a+(j-1)\Delta x$ for $j=1,\dots,N$, where $\Delta x=(b-a)/(N-1)$. Let $C$ be a positive integer. Suppose that $\widetilde{F}(x)$ is a $C^{d}$-periodic extension of $\widetilde{f}(x)$ with period $P=b-a+\ell$, where $\ell=C\Delta x$. The derivatives of $f(x)$ are approximated by spectral differentiation of $\widetilde{F}(x)$.

Denote by $\mathrm{i}$ the imaginary unit, set $M=N+C-1$ and let $\mathbf{f}^E=(f_1^E,\dots,f_M^E)^T$,
 where $f_j^E=f_j$ for $j=1,\dots,N$ and $f_j^E=f_H(x_j)$ for $j=N+1,\dots,M$. The $k$-th derivative is obtained by a real FFT (RFFT) of $\mathbf{f}^{E}$, multiplication of Fourier mode $j$ by $(2\pi\mathrm{i}j/P)^{k}$, and an inverse real FFT (IRFFT)~\cite{Dahlquist2008}. A spectral filter $\varsigma(x)$ can be applied by replacing $\widehat{f}_{j}$ by $\varsigma(2j/M)\widehat{f}_{j}$ \cite[Section~8.1]{Canutoetal2007}.

The two following tests illustrate FFT-based differentiation using the values $f_{j}^{E}$ generated by \Cref{alg:extrablend}.

\textbf{\em{Test 1.}} The derivatives of $f(x)=\cos((1+x)^{2})$ are approximated at equidistant points of $[0,1]$ by FFT-based differentiation using EB with $C=64$, $d=9$ and $\sigma(x)=T_{11}(x)$ in \cref{alg:extrablend}. The approximate $C^{d}$-periodic extension of $f(x)$ is shown in \cref{fig:test1}. \Cref{tab:highderiv} compares the $L^{\infty}$ errors for the first, second, and fifth derivatives with FC-Gram~\cite{bruno2010,AmlaniBruno2016} ($C=25$, $d=9$) and FOD/NOD~\cite{Egidi2023}: EB is at least as accurate as FC-Gram, and more accurate beyond the coarsest resolution, while generally outperforming FOD/NOD before roundoff effects become visible.

\begin{figure}[!ht]
	\centering
	\includegraphics[width=.35\textwidth]{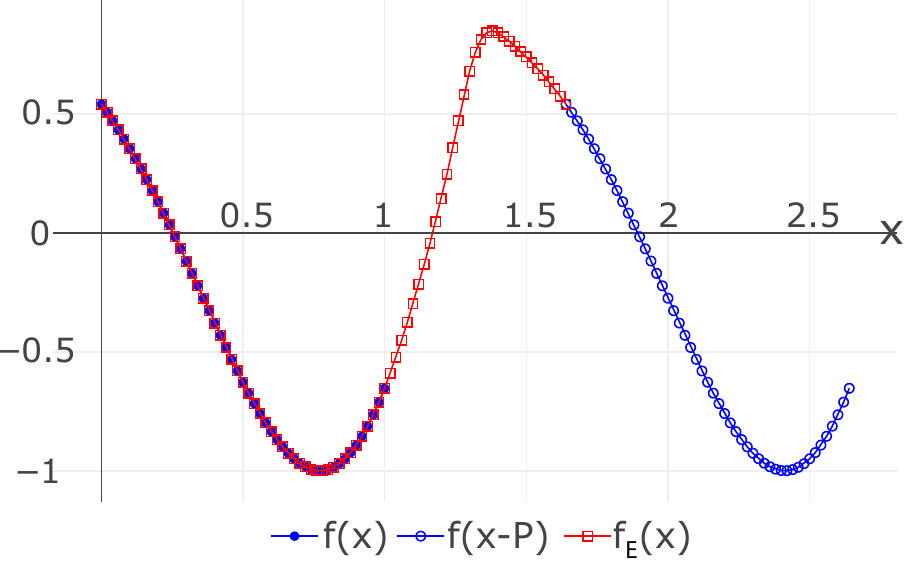}
  \caption{Construction of a $C^{9}$-periodic extension of $f(x)=\cos((1+x)^2)$, with period $P=1.64$.}
  \label{fig:test1}
\end{figure}

\begin{table}[!htbp]
\small
\centering
\caption{$L^{\infty}$ errors for FFT-based differentiation of $f(x)=\cos((1+x)^{2})$. FOD/NOD values are reproduced directly from~\cite{Egidi2023}.}
\begin{tabular*}{\tblwidth}{@{}C|RRR|RRR|RRR@{}}
 \toprule
 &  & $f'(x)$ & & & $f''(x)$ & & & $f^{(5)}(x)$ & \\
 $1/\Delta x$ & EB & FC-Gram & FOD~\cite{Egidi2023} & EB & FC-Gram & NOD~\cite{Egidi2023} & EB & FC-Gram & NOD~\cite{Egidi2023} \\
 \midrule
 25 & 1.07(-7) & 1.07(-7) & 1.07(-5) & 3.24(-6) & 3.24(-6) 
 & 6.04(-5) & 1.09(-2) & 1.09(-2) & 1.90(-2) \\
 50 & 7.96(-10) & 1.06(-9) & 6.69(-7) & 4.77(-8) & 4.82(-8) 
 & 6.69(-6) & 1.17(-3) & 1.71(-3) & 1.52(-3) \\
 100 & 3.69(-12) & 7.57(-10) & 4.18(-8) & 4.59(-10) & 2.89(-9) 
 & 5.15(-7) & 8.45(-5) & 2.48(-2) & 3.63(-4) \\
 200 & 4.85(-13) & 2.71(-9) & 2.62(-9) & 9.39(-11) & 2.24(-8) 
 & 3.52(-8) & 1.40(-4) & 1.41(+0) & 8.61(-3) \\
 400 & 1.55(-12) & 4.05(-9) & 1.64(-10) & 4.94(-10) & 6.80(-8)
 & 2.42(-9) & 6.84(-3) & 3.35(+1) & 1.95(-1) \\ 
 800 & 5.42(-12) & 6.48(-9) & 1.07(-11) & 5.13(-9) & 2.10(-7) 
 & 1.69(-9) & 1.79(-1) & 8.62(+3) & 8.58(+0) \\
 1600 & 5.62(-12) & 1.42(-8) & 3.05(-12) & 7.29(-9) & 1.25(-6) 
 &  7.98(-9) & 5.83(+0) & 2.97(+4) & 3.25(+2) \\
 \bottomrule
\end{tabular*}
\label{tab:highderiv}
\end{table}
% \enlargethispage{2\baselineskip}
\textbf{\em{Test 2.}} Consider the following 1D linear advection problem:
\begin{equation}\label{eq:1Dadvection}
  \begin{cases}
 \dfrac{\partial u}{\partial t} + c\dfrac{\partial u}{\partial x} = 
 \left(\dfrac{c}{1.15-x} - 1\right)\phi(x,t), \quad & 0\leq x\leq 1, \quad t\geq 0,\\[.75em]
 	u(0,t)  = -\sin(4\pi ct) + \phi(0,t), \quad & t\geq 0,\\
  u(x,0)  = \sin(4\pi x) + \phi(x,0), \quad & 0\leq x\leq 1,
  \end{cases}
\end{equation}
% subject to the following conditions
% \begin{align}
% 	u(0,t) & = -\sin(4\pi c) + \phi(0,t), 
%   u(x,0) & = \sin(4\pi x) + \phi(x,0), \label{eq:ic}
% \end{align}
where $\phi(x,t)=e^{-t}/(1.15-x)$ and $c>0$. For $c=1$, the problem is integrated to $t=1$ using the fourth-order Adams-Bashforth method: $u_{k} = u_{k-1} +  (55G_{k-1} - 59G_{k-2} + 37G_{k-3} -9G_{k-4})\Delta t/24$, where $G_{k}=-c\frac{\partial u_{k}}{\partial x} + \left(\frac{c}{1.15-x} - 1\right)\phi(x,t_k)$, $u_{k} = u(x,t_{k})$, $t_{k}=k\Delta t$, and $\Delta t=0.1\Delta x/c$ based on the CFL condition~\cite[Appendix~D.2.2]{Canutoetal2007}. The exact solution $u(x,t)=\sin(4\pi(x-ct))+\phi(x,t)$ initializes $u_0,\dots,u_3$. At the $k$-th step, periodic extensions of $u_{k-1}$ are generated by EB taking $d=5$, $C\in\{25,64\}$ and $\sigma(x)=T_{11}(x)$. The spatial derivatives required to evaluate $G_{k-1}$ are computed by FFT-based differentiation of the extensions using the filter $\varsigma(x)=\exp(-\alpha x^{\rho})$ with $\alpha=2\log(10)$ and $\rho=16$. Then $u_{k}$ is updated using the filtered solution $u_{k-1}$ reconstructed by IRFFT.

Let $e_{k}$ be the $L^{\infty}$ error up to $t=1$ for $N_{k}+1$ points, where $N_{k}=2^{4+k}$ and $\Delta x=1/N_{k}$ for $k=1,\dots,8$. For $k=2,\dots,8$, define $O_{k}=\log(e_{k-1}/e_{k})/\log(N_{k}/N_{k-1})$. Convergence rates for the solutions generated by EB and FC-Gram ($C=25, d=5$) are shown in \Cref{tab:1Dadvsol}. For EB with $C=64$, the observed convergence rates remain between approximately $4.6$ and $4.9$ before the error floor is reached.

\begin{table}[!htbp]
\small  
\centering
\caption{Maximum ($L^\infty$) error 
over all space and time of the numerical solution of the system given by
\cref{eq:1Dadvection}.}
\begin{tabular*}{\tblwidth}{@{}R|RRRRRR@{}}
  \toprule
$1/\Delta x$ & $e_{k}$ (EB), $C=64$ & $O_{k}$ (EB), $C=64$
& $e_{k}$ (EB), $C=25$ & $O_{k}$ (EB), $C=25$
& $e_{k}$ (FC-Gram) & $O_{k}$ (FC-Gram) \\
  \midrule
  32 & 1.11(-3) & - & 1.1(-3) & - &  1.11(-3) & - \\
  64 & 3.71(-5) & 4.90 & 3.55(-5) & 4.96 & 3.70(-5) & 4.90 \\
  128 & 1.51(-6) & 4.62 & 1.51(-6) & 4.55 & 1.51(-6) & 4.61 \\
  256 & 5.91(-8) & 4.67 & 1.07(-7) & 3.81 & 5.94(-8) & 4.67 \\
  512 & 2.08(-9) & 4.83 & 8.15(-8) & 3.95(-1) & 2.29(-9) & 4.69 \\
  1024 & 6.93(-11) & 4.91 & 6.57(-8) & 3.11(-1) & 1.44(-9) & 6.65(-1) \\
  2048 & 2.62(-12) & 4.72 & 5.87(-8) & 1.68(-1) & 1.44(-9) & 6.18(-3) \\
  4096 & 2.08(-12) & 3.36(-1) & 5.56(-8) & 7.97(-2) & 1.44(-9) & 6.54(-4) \\
  \bottomrule
\end{tabular*}
\label{tab:1Dadvsol}
\end{table}

%%%%%%%%%%%%%%%%%%%%%%%%%%%%%%%%%%%% new section %%%%%%%%%%%%%%%%%%%%%%%%%%%%%%%%%%%%%%%%%%%%
\section{Conclusions}
\label{sec:conclusions}

EB is a matrix-free method for generating approximate smooth periodic extensions from equidistant samples on the fly, providing an alternative to precomputed Fourier continuation constructions for FFT-based differentiation of non-periodic functions. Further work will address error bounds and the selection of the smooth step function and continuation parameters.

\subsection*{Acknowledgments}

This work was supported by the Agence Nationale de la Recherche (ANR) under Grant ANR-23-CE46-0008.

\subsection*{Data availability}

A companion Pluto notebook implemented in Julia is available at \href{https://pluto.land/n/ptc1yrux}{https://pluto.land/n/ptc1yrux}.


\bibliographystyle{elsarticle-num-names}

\bibliography{references}

\begin{thebibliography}{16}
\expandafter\ifx\csname natexlab\endcsname\relax\def\natexlab#1{#1}\fi
\providecommand{\url}[1]{\texttt{#1}}
\providecommand{\href}[2]{#2}
\providecommand{\path}[1]{#1}
\providecommand{\DOIprefix}{doi:}
\providecommand{\ArXivprefix}{arXiv:}
\providecommand{\URLprefix}{URL: }
\providecommand{\Pubmedprefix}{pmid:}
\providecommand{\doi}[1]{\href{http://dx.doi.org/#1}{\path{#1}}}
\providecommand{\Pubmed}[1]{\href{pmid:#1}{\path{#1}}}
\providecommand{\bibinfo}[2]{#2}
\ifx\xfnm\relax \def\xfnm[#1]{\unskip,\space#1}\fi
%Type = Article
\bibitem[{Boyd(2002)}]{Boyd2002}
\bibinfo{author}{J.~P. Boyd},
\newblock \bibinfo{title}{A comparison of numerical algorithms for {Fourier}
  extension of the first, second, and third kinds},
\newblock \bibinfo{journal}{J. Comput. Phys.} \bibinfo{volume}{178}
  (\bibinfo{year}{2002}) \bibinfo{pages}{118--160}.
%Type = Article
\bibitem[{Huybrechs(2010)}]{Huybrechs2010}
\bibinfo{author}{D.~Huybrechs},
\newblock \bibinfo{title}{On the {Fourier} extension of nonperiodic functions},
\newblock \bibinfo{journal}{SIAM J. Numer. Anal.} \bibinfo{volume}{47}
  (\bibinfo{year}{2010}) \bibinfo{pages}{4326--4355}.
%Type = Article
\bibitem[{Bruno and Lyon(2010)}]{bruno2010}
\bibinfo{author}{O.~P. Bruno}, \bibinfo{author}{M.~Lyon},
\newblock \bibinfo{title}{High-order unconditionally stable {FC-AD} solvers for
  general smooth domains {I}. {Basic} elements},
\newblock \bibinfo{journal}{J. Comput. Phys.} \bibinfo{volume}{229}
  (\bibinfo{year}{2010}) \bibinfo{pages}{2009--2033}.
%Type = Article
\bibitem[{Albin and Pathmanathan(2014)}]{AlbinPathmanathan2014}
\bibinfo{author}{N.~Albin}, \bibinfo{author}{S.~Pathmanathan},
\newblock \bibinfo{title}{Discrete periodic extension using an approximate step
  function},
\newblock \bibinfo{journal}{SIAM J. Sci. Comput.} \bibinfo{volume}{36}
  (\bibinfo{year}{2014}) \bibinfo{pages}{A668--A692}.
%Type = Article
\bibitem[{Amlani and Bruno(2016)}]{AmlaniBruno2016}
\bibinfo{author}{F.~Amlani}, \bibinfo{author}{O.~P. Bruno},
\newblock \bibinfo{title}{An {FC}-based spectral solver for elastodynamic
  problems in general three-dimensional domains},
\newblock \bibinfo{journal}{J. Comput. Phys.} \bibinfo{volume}{307}
  (\bibinfo{year}{2016}) \bibinfo{pages}{333--354}.
%Type = Article
\bibitem[{Amlani et~al.(2023)Amlani, Wei, and Pahlevan}]{amlani2023fourier}
\bibinfo{author}{F.~Amlani}, \bibinfo{author}{H.~Wei}, \bibinfo{author}{N.~M.
  Pahlevan},
\newblock \bibinfo{title}{A {Fourier}-based methodology without numerical
  diffusion for conducting dye simulations and particle residence time
  calculations},
\newblock \bibinfo{journal}{J. Comput. Phys.} \bibinfo{volume}{493}
  (\bibinfo{year}{2023}) \bibinfo{pages}{112472}.
%Type = Article
\bibitem[{Nainwal and Anand(2025)}]{NainwalAnand2025}
\bibinfo{author}{P.~Nainwal}, \bibinfo{author}{A.~Anand},
\newblock \bibinfo{title}{A modified {FC-Gram} approximation algorithm with
  provable error bounds},
\newblock \bibinfo{journal}{J. Sci. Comput.} \bibinfo{volume}{105}
  (\bibinfo{year}{2025}).
%Type = Article
\bibitem[{Nainwal and Anand(2026)}]{NainwalAnand2026}
\bibinfo{author}{P.~Nainwal}, \bibinfo{author}{A.~Anand},
\newblock \bibinfo{title}{A generalized {FC-Gram} approximation framework with
  analysis and applications},
\newblock \bibinfo{journal}{arXiv preprint arXiv:2605.04765}
  (\bibinfo{year}{2026}).
%Type = Article
\bibitem[{Boyd(2006)}]{Boyd2006}
\bibinfo{author}{J.~P. Boyd},
\newblock \bibinfo{title}{Asymptotic {Fourier} coefficients for a ${C^\infty}$
  bell (smoothed-“top-hat”) \& the {Fourier} extension},
\newblock \bibinfo{journal}{J. Sci. Comput.} \bibinfo{volume}{29}
  (\bibinfo{year}{2006}) \bibinfo{pages}{1--24}.
%Type = Book
\bibitem[{Laks{\aa}(2022)}]{Laksa2022}
\bibinfo{author}{A.~Laks{\aa}}, \bibinfo{title}{Blending techniques in Curve
  and Surface construction}, \bibinfo{publisher}{GEOFO}, \bibinfo{year}{2022}.
%Type = Article
\bibitem[{M\'{e}ndez-Cruz and Amlani(2026)}]{MendezFaisal2026}
\bibinfo{author}{I.~M\'{e}ndez-Cruz}, \bibinfo{author}{F.~Amlani},
\newblock \bibinfo{title}{Blend-to-zero operators for smooth transition
  functions},
\newblock \bibinfo{journal}{arXiv preprint arXiv:2604.26440}
  (\bibinfo{year}{2026}).
%Type = Article
\bibitem[{Olofsen(2019)}]{Olofsen2019}
\bibinfo{author}{H.~Olofsen},
\newblock \bibinfo{title}{Blending functions based on trigonometric and
  polynomial approximations of the {Fabius} function},
\newblock \bibinfo{journal}{Norsk Informatikkonferanse} \bibinfo{volume}{2019}
  (\bibinfo{year}{2019}) \bibinfo{pages}{12s}. \URLprefix
  \url{https://hdl.handle.net/10037/17358}.
%Type = Article
\bibitem[{Webb et~al.(2012)Webb, Trefethen, and Gonnet}]{Webbetal2012}
\bibinfo{author}{M.~Webb}, \bibinfo{author}{L.~N. Trefethen},
  \bibinfo{author}{P.~Gonnet},
\newblock \bibinfo{title}{Stability of barycentric interpolation formulas for
  extrapolation},
\newblock \bibinfo{journal}{SIAM J. Sci. Comput.} \bibinfo{volume}{34}
  (\bibinfo{year}{2012}) \bibinfo{pages}{A3009--A3015}.
%Type = Book
\bibitem[{Dahlquist and Bj\"{o}rck(2008)}]{Dahlquist2008}
\bibinfo{author}{G.~Dahlquist}, \bibinfo{author}{A.~Bj\"{o}rck},
  \bibinfo{title}{Numerical Methods in Scientific Computing, Volume I},
  \bibinfo{publisher}{SIAM}, \bibinfo{year}{2008}.
%Type = Book
\bibitem[{Canuto et~al.(2007)Canuto, Quarteroni, Hussaini, and
  Zang}]{Canutoetal2007}
\bibinfo{author}{C.~Canuto}, \bibinfo{author}{A.~Quarteroni},
  \bibinfo{author}{M.~Y. Hussaini}, \bibinfo{author}{T.~A. Zang},
  \bibinfo{title}{Spectral Methods: Evolution to Complex Geometries and
  Applications to Fluid Dynamics}, \bibinfo{publisher}{Springer-Verlag},
  \bibinfo{year}{2007}.
%Type = Article
\bibitem[{Egidi et~al.(2023)Egidi, Giacomini, Maponi, and Youssef}]{Egidi2023}
\bibinfo{author}{N.~Egidi}, \bibinfo{author}{J.~Giacomini},
  \bibinfo{author}{P.~Maponi}, \bibinfo{author}{M.~Youssef},
\newblock \bibinfo{title}{An {FFT} method for the numerical differentiation},
\newblock \bibinfo{journal}{Appl. Math. Comput.} \bibinfo{volume}{445}
  (\bibinfo{year}{2023}) \bibinfo{pages}{127856}.

\end{thebibliography}


\end{document}